\documentclass[11pt]{article}\usepackage{amsmath, amscd, amsthm} 
\usepackage{amssymb, amsfonts}
\input xy \xyoption{all}
\CompileMatrices
\numberwithin{equation}{subsection}%

\def\proof{\par\medskip\noindent {\bf Proof: }}
\def\qed{\hfill $\Box$ \medskip \par}
\def\Box{\framebox[10pt]{\rule{0pt}{3pt}}}

\ifx\ThmNum\undefined 
   \newtheorem{theorem}{Theorem}[section]
   
\else 
   \newtheorem{theorem}{Theorem} 
   
\fi

\newcommand{\R}{\mathbb{R}}

\newtheorem{remark}[theorem]{Remark}
\newtheorem{proposition}[theorem]{Proposition}
\newtheorem{lemma}[theorem]{Lemma}
\newtheorem{definition}[theorem]{Definition}
\newtheorem{corollary}[theorem]{Corollary}

\title{ The cube axiom and resolutions in homotopy theory}
\author{ Manfred Stelzer}
\date{\today}

\begin{document}

\maketitle

\begin{abstract}
We show that  a version of the cube axiom holds in cosimplicial unstable coalgebras and cosimplicial spaces equipped with a resolution model structure. As an application, classical theorems in unstable homotopy theory are extended to this context.
\end{abstract}

\section{Introduction}
 The cube axiom holds in a model category  $\mathcal{M}$ if in every commutative cube  the top square is a homotopy pushout if the bottom square is a homotopy pushout and the vertical faces are homotopy pullbacks.   It was proved by Mather that  the cube axiom is valid in topological spaces with the Str\o m structure \cite{M}. In many other model categories it is known to  hold.
 Examples include every stable model category, every $\infty$-topos and the model category of  
 motivic spaces \cite[Proposition 3.15]{Ho}.

 In the model categorical axiomatic  approach to Lusternik-Schnirelmann category, due to Doerane, a variant of the cube axiom is to be assumed \cite{Do},\cite{Do2}. If  it is satisfied many of the classical topological results on Lusternik-Schnirelmann category extend. \\
To know that the cube axiom holds brings some more benefits.
There is a good theory of fiberwise localization \cite{C.S}.
 Moreover, in a recent paper  of Devalpurka and Haine it is shown that a form of the James splitting and a weak form of the Hilton-Milnor theorem hold under some   assumptions in an $\infty$-category \cite{D.H}. The main assumption  made  is the cube axiom.\\ 

 In this paper, we investigate in which form the cube axiom holds in 
two resolution model categories of cosimplicial objects. The first is the one of cosimplicial spaces $c\mathcal{S}^{\mathcal{G}}$ and the second the one of cosimplicial unstable coalgebras $c\mathcal{CA}^{\mathcal{E}}$   for a prime field $\mathbb{F}$.   The classes of group objects $\mathcal{G}$ and $\mathcal{E}$ which define the resolutions are the Eilenberg-MacLane spaces of type  $\mathbb{F}$ and their homology  respectively \cite{Bou}. These model categories give home to the constructions  of unstable Adams spectral sequences of Bousfield and Kan \cite{Bou2},\cite{BK}. For this reason, they also  play a central role in 
the obstruction theory for realizations of unstable coalgebras  and the construction of moduli spaces for such realizations  \cite{Bl}, \cite{BRS}.
 Since  $c\mathcal{CA}^{\mathcal{E}}$ and $c\mathcal{S}^{\mathcal{G}}$  are not known  to be cofibrantly generated, the general results of 
 \cite{Re} and  \cite{R.S.} do not apply.
     In fact, it is the class of fibrations which is defined by cocell attachments.\\
       As an application we see that versions of  classical theorems of James, Hilton-Milnor and Ganea hold in   $c\mathcal{CA}^{\mathcal{E}}$ and  $c\mathcal{S}^{\mathcal{G}}$.\\
   For the rest of this paper $\mathcal{M}$ stands for 
 $\mathcal{CA}$ or $\mathcal{S}$  and $\mathcal{F}$ for $\mathcal{E}$ or $\mathcal{G}$.\\
We show the following:

 \begin{theorem}\label{cube}
  The cube axiom holds in  $c\mathcal{M}^{\mathcal{F}}$
 for cubes  
  \[
\xymatrix@!0{
& E_A \ar[dl]\ar'[d][dd]
& &  E_B \ar[dd]\ar[dl]\ar[ll]
\\
E_D \ar[dd]
& & E_C \ar[ll]\ar[dd]
\\
& A \ar[dl]
& & B \ar[dl]\ar'[l][ll]
\\
D 
& & C\ar[ll] 
}
\]

   in which
  the map 
  $B\to A$ is $0$-connected. 

 \end{theorem}
 
\begin{remark} That one has to make assumptions on the maps in the cube   is not uncommon. Several examples of categories in which such a restriction is necessary for the cube axiom or its dual to hold can be found in \cite[Appendix]{Do}. A particular important example is the  category of commutative differential graded algebras over the rationals. This serves  as one model for rational homotopy theory.
\end{remark}
 
   Let $Z $ be a fibrant and cofibrant object in $c\mathcal{CA}^{\mathcal{E}}$ or $ c\mathcal{S}^{\mathcal{G}}$ and write $c\mathcal{CA}^{\mathcal{E}}_Z$ and  $ c\mathcal{S}^{\mathcal{G}}_Z$ for the model categories of objects over $Z$ with coaugmentation under  $Z$. These are pointed model categories with $0$-object $Z$.

 All suspensions, loop spaces, products, smash and wedge products in  $c\mathcal{M}_Z^{\mathcal{F}}$ are taken in the derived sense and over and under the $0$-object  $Z$.
  \\Combining \ref{cube} with some of the main theorems in \cite{D.H} we obtain the following  application.\\


 \begin{theorem}\label{main application}  Let $X$, $Y$  be  objects in $c\mathcal{M}_Z^{\mathcal{F}}$.
  Assume that the structure maps $X\to Z$ and $Y\to Z$ are $0$-connected.
 Then there are  equivalences in $c\mathcal{M}^{\mathcal{F}}$
 
 \begin{itemize} 
 
\item[(1)] $\Sigma(X\times Y)\simeq (\Sigma X\wedge Y)\vee (\Sigma X)\vee (\Sigma Y)
 $\\
 \item[(2)] (James splitting theorem)
 $\Sigma\Omega \Sigma X\simeq \bigvee_{i\geq 1} \Sigma X^{\wedge i}$\\

 \item[(3)](Hilton-Milnor splitting )\\
  $\Omega\Sigma  (X\vee Y)\simeq
\Omega\Sigma X\times \Omega\Sigma Y\times \Omega (\bigvee_{i,j\geq 1}\Sigma X^{\wedge i}\wedge Y^{\wedge j}) 
 $ 
 \end{itemize}
 and  there are  homotopy fiber sequences
 \begin{itemize}

\item[(4)] $\Sigma\Omega X\to X\vee X\to X$\\

\item[(5)] $\Sigma(\Omega X \wedge \Omega Y) \to X\vee Y \to X \times Y$  \\
\item[(6)] $(\Omega Y\times X)/\Omega Y\to X\vee Y\to Y$\\
\end{itemize}
Let $F\to E\to X$ be a fibration sequence and  write $E\cup CF$ for the homotopy cofiber of the inclusion of the homotopy fiber $F$ in $E$.  Assume that  $F\to Z$ or $F\to E$ is $0$-connected. Then there is a homotopy fiber sequence
\begin{itemize}
\item[(7)] (Ganea's theorem ) $\Sigma F\wedge \Omega X\to E\cup CF\to X.$\\
\end{itemize}


 \end{theorem}
 
In  the usual form of the Hilton-Milnor theorem the right hand side  of (3) in  \ref{main application} is further decomposed into an infinite product. We prove such a version in \ref{traditional HM} below.  This needs a convergence property which is implied by some connectivity estimates which are derived  from the dual Blakers-Massey theorem in \cite{BRS}.   

\begin{theorem}\label{traditional HM}  Let $X$, $Y$  be as in \ref{main application}.
 Then there is an equivalence in $c\mathcal{M}^{\mathcal{F}}$
\[\Omega\Sigma  (X\vee Y)\simeq \prod_{w} \Omega \Sigma w(X,Y)\]
where $w$ runs through all basic words in $(x,y)$ and $w(X,Y)$ is the iterated smash power of $X,Y$ defined by $w$.

\end{theorem}

 \begin{remark} Under the additional assumptions that the homotopy groups 
 $\pi_* (X),\pi_* (Y)$ are free  $\pi_0 (Z)$-modules and that $\pi_{*}(Z)$ is concentrated in degree $0$ a (dual) Hilton-Milnor theorem was established by Goerss in the category of simplicial unstable algebras \cite{G}. He noted that possibly this restriction may be omitted by restructuring the proof along the lines of Milnor's  original argument \cite{Mi}, \cite{Wh}. That is what we did. Computational applications of  his theorem  where given  in \cite{G2} and \cite{G3}. Many of them generalize due to \ref{traditional HM}. In particular,  \ref{traditional HM} helps in the computation of Andr\'{e}-Quillen cohomology of coabelian objects in $c\mathcal{CA}^{\mathcal{E}}$ which shows up in the $E_2$-term of the unstable Adams spectral sequence.
\end{remark}

\begin{remark} Many of our arguments extend to resolution model categories defined  by the spaces in the $\Omega$ spectrum of Morava K-theory. The unstable homology operations are also known due to the determination of the Hopf ring of Morava K-theory by Wilson \cite{W}. One ingredient which is missing so far is the homotopy excision theorem.
\end{remark}

The paper is organized as follows. In section 2 we give background information on  the model categories $c\mathcal{CA}^{\mathcal{E}}$ and $c\mathcal{S}^{\mathcal{G}}$.  The proof of \ref{cube} is given  in section 3. Section 4 is devoted to the proof of \ref{main application}. Connectivity estimates for relative infinite products, wedge and smash products are derived in section 5. These are needed in the proof of
\ref{traditional HM}  which is the content of section 6.

\subsection{Conventions} Throughout the paper we fix a prime field $\mathbb{F}$.   We will use the following notation. 
\begin{itemize}
\item $\mathcal{S}$ = the category of simplicial sets;
\item $Vec$ = the category of non-negatively graded $\mathbb{F}$-vector spaces
\item  $\mathcal{CA}$ = the category of unstable coalgebras over the Steenrod algebra $\mathcal{A}$ over $\mathbb{F}$
\end{itemize}
For any category $\mathcal{C}$ we let
\begin{itemize}
\item  $c\mathcal{C}$ = the category of cosimplicial objects over  $\mathcal{C}$;
\end{itemize}
We use freely notions and standard facts about model categories. Besides the founding \cite{Q} some widely used  references are \cite{Hi},\cite{H} and \cite{GJ}.
In section 5 we use the language of $\infty$-categories. The needed definitions and results are to be found in \cite{L}.

\subsection{Acknowledgments}
I would like to thank Hadrian Heine and Markus Spitzweck for patiently answering my questions on $\infty$-categorical issues.
Finally, I would like to thank the Deutsche Forschungsgemeinschaft for
support through the Schwerpunktprogramm 
1786 ``Homotopy theory and algebraic geometry''.
 


\section{Cosimplicial spaces and unstable coalgebras} \label{Res}


\subsection{Resolution model categories}
The categories $\mathcal{S}$  and $\mathcal{CA}$  are equipped with the standard Quillen and the discrete model structure respectively.  
 Let $K(\mathbb{F},m)$ denote the Eilenberg-MacLane space of type $\mathbb{F}$. A product of those spaces will be called an $\mathbb{F}$-gem. Let $\mathcal{G}=\{K(\mathbb{F},m)|m\geq 0 \}$ 
 and $\mathcal{E}=\{H_* (K(\mathbb{F},m))|m\geq 0\}.$  The categories $c\mathcal{S}$ of cosimplicial spaces and cosimplicial unstable coalgebras $c\mathcal{CA}$  carry simplicial  resolution model category structures relative
to $\mathcal{G}$ and $\mathcal{E}$ respectively \cite{D.K.S.}, \cite{Bou}.  We explain the basic facts about the model categories $c\mathcal{S}^{\mathcal{G}}$ and $c\mathcal{CA}^{\mathcal{E}}$. As any resolution model category they carry an external simplicial structure.\\ 
The weak equivalences are the maps of cosimplicial objects 
\begin{equation}X\to Y\end{equation} which induce an isomorphism on $\pi_* [-,F]$ for every $F\in  \mathcal{F}$.
 This can be equivalently described as maps which induce an isomorphism on $\pi^* H_* (-)$ and $\pi^* (-)$ respectively. For  $\mathcal{E}$ this equivalence is due to the fact that objects in $\mathcal{E}$   are cofree unstable coalgebras. 

 The cofibrations are the Reedy cofibrations such that the induced homomorphism\begin{equation} [Y,F]\to[X,F]\end{equation} is a fibration of simplicial groups. A concrete description of the fibrations  can be found and will be of importance later on.
 \begin{definition}A map $i:A\to B$ in the homotopy category 
 $Ho(\mathcal{M})$ is called $\mathcal{F}$-monic if \[i^* :[B,F]\to [A,F]\]
 is surjective for al $F\in \mathcal{F}$.\\
 An object $Y\in Ho(\mathcal{M})$ is called $\mathcal{F}$-injective
 if 
 \[i^* :[B,Y]\to [A,Y]\] is surjective for all $\mathcal{F}$-monic 
 $i:A\to B$. 
 \end{definition}
 For $D\in c\mathcal{M}$  and $K\in \mathcal{S}$ we write $D\otimes K$ and $D^K$ for the tensor and cotensor.  Now we can describe the fibrations in $c\mathcal{M}^{\mathcal{F}}$. Recall from \cite[2.3.]{BRS}.



\begin{definition}\label{co-cell decomposition of cofree maps}
A map $f:X\to Y$ in $c\mathcal{M}$ is quasi-$\mathcal{F}$-cofree if there is a sequence of fibrant $\mathcal{F}$-injective objects $(F_s)_{s\ge 0}$ such that  for all $s\ge 0$, there are homotopy  pullback diagrams in the Reedy model structure of $c\mathcal{M}$
\[ 
\xymatrix{
 cosk_s(f) \ar[r]\ar[d]_{\gamma_{s-1}(f)} & (cF_s)^{\Delta^s} \ar[d]^{i_s^*} \\
        cosk_{s-1}(f) \ar[r]   & (cF_s)^{\partial\Delta^s}}\]
where $i_s :\partial\Delta^s\to\Delta^s$ is the canonical inclusion.

\end{definition}

 \begin{definition}\label{co-cell } Let $F$ be a fibrant $\mathcal{F}$-injective. A  a Reedy homotopy pullback diagram in $c\mathcal{M}$
 \[ 
\xymatrix{
 E \ar[r]\ar[d]_{p} & (cF)^{\Delta^s} \ar[d]^{i_s^*} \\
        D \ar[r]^{f}   & (cF)^{\partial\Delta^s}}\]
        is called a   quasi-cocell attachment of dimension $s$.
        The map $f$ is called the coattaching map.
 \end{definition}
 A map in $c\mathcal{M}^{\mathcal{F}}$  is a fibration if and only if it is a Reedy fibration and a retract of a  quasi-cofree map \cite[corollary 2.3.12]{BRS}.
 We  recall the definition of the external loop functor.
 
\begin{definition}\label{external-loops}
The collapsed boundary gives the $n$-sphere $S^n:=\Delta^n/\partial\Delta^n$ a canonical basepoint.
If $X$ is pointed in $c\mathcal{M}$ we define the  $s$-th external loop object by
   $$ \Omega^n_{ ext}(X)=\overline{hom}(\Delta^n/\partial\Delta^n, X) $$ where  
  $$ \overline{hom}( \Delta^n/\partial\Delta^n,X) = fib[hom(\Delta^n/\partial\Delta^n,X)\to\hom(\Delta^0,X)\cong Y]\in c\mathcal{M}_* $$
as the fiber taken at the basepoint of $X$ of the map induced by the basepoint of $\Delta^n/\partial\Delta^n$.\end{definition}
   
We will often omit the subscript ext when the reference to the external structure is clear from the context.\\

 The following observation from \cite{BRS} will be useful to us later on. Let $E$ be as in \ref{co-cell }. Then  each cosimplicial degree $t$ there is an isomorphism
\begin{equation}E^t \cong D^t \times \Omega^s (cF)^t. 
\end{equation}

 As a consequence of homotopy excision the model category $c\mathcal{M}^{\mathcal{F}}$ is proper \cite[4.4.2 6.2.5]{BRS}.\\
  We recall the notion of cosimplicial connectivity.

 \begin{definition}\label{def. n-cocon.}
Let $n \geq 0$. 
\begin{itemize}
\item[(a)]A map
   $f:X\to Y$ in $c\mathcal{M}^{\mathcal{F}}$ is called $n$-connected if the induced map
  \[ \pi_{s}[Y,F]_{\mathcal{M}}\to\pi_{s}[X,F]_{\mathcal{M}} \]
is an isomorphism for $0 < s < n$ and an epimorphism for $s=n$ for all $F
\in \mathcal{F}$
\item[(b)] A {\it pointed object $(C, \ast)$  of $c\mathcal{M}^{\mathcal{F}}$ is cosimplicially $n$-connected} if the map $\ast \to C$ is cosimplicially $n$-connected.
\item[(c)] An arbitrary object $X$ of $c\mathcal{M}^{\mathcal{F}}$ is called   $n$-connected if the canonical map $X \to \ast$ is cosimplicially $n$-connected.
\end{itemize}
\end{definition}

A pointed object $(C, \ast)$  of $c\mathcal{M}^{\mathcal{F}}$ is  $n$-connected if and only if $C$ is cosimplicially $(n-1)$-connected.
Every pointed object is cosimplicially $0$-connected.\\

The next result from \cite{BRS} is used to reduce questions     in $c\mathcal{S}^{\mathcal{G}}$ to questions in the simpler category $c\mathcal{CA}^{\mathcal{E}}$.
 

\begin{lemma}\label{Hsquare}The homology functor $H_*: c\mathcal{S}^{\mathcal{G}} \to c\mathcal{CA}^{\mathcal{E}}$  sends homotopy pullbacks into homotopy pullbacks and homotopy puhouts along maps one of which is $0$-connected to homotopy pushouts \cite[Proposition 6.17.,Corollary 6.2.3.]{BRS}.\end{lemma}

Objects in $\mathcal{M}$ satisfy the following $\mathcal{F}$- flatness property with respect to the  product.
 
\begin{lemma} Let $X\in \mathcal{M}$ and $i:A\to B$ an $\mathcal{F}$-injective map in 
$Ho(\mathcal{M})$. Then $i\times 1_X:A\times X \to B\times X$ is   $\mathcal{F}$-injective.
\end{lemma}
\begin{proof} Suppose first that   $\mathcal{M}=\mathcal{CA}$. In this case a map is  $\mathcal{E}$-injective if and only if it is injective. This follows from the fact that the homology of an Eilenberg-MacLane space of type $\mathbb{F}$ is cofree. The product in $\mathcal{CA}$ is the tensor product which is an exact functor of the underlying graded vector spaces.\\
In case $\mathcal{M}=\mathcal{S}$ note that a map $i:A\to B$ in  $Ho(\mathcal{S})$ is $\mathcal{G}$-injective if and only if $H_* (i)$ is $\mathcal{E}$-injective in  $Ho(\mathcal{CA})$ \cite[p.68]{BRS}.$\square$ 
\end{proof}


\section{The cube theorem in  $c\mathcal{CA}^{\mathcal{E}}$ and  $c\mathcal{S}^{\mathcal{G}}$}

\begin{definition}\label{cube def} Let  $\mathcal{N}$ be a model category. We say that the cube axiom holds in $\mathcal{N}$ if in any commutative cube in $\mathcal{N}$
      \[
\xymatrix@!0{
& E_A \ar[dl]\ar'[d][dd]
& &  E_B \ar[dd]\ar[dl]\ar[ll]
\\
E_D \ar[dd]
& & E_C \ar[ll]\ar[dd]
\\
& A \ar[dl]
& & B \ar[dl]\ar'[l][ll]
\\
D 
& & C\ar[ll] 
}
\]
 whose bottom square is a homotopy pushout and whose vertical squares are homotopy pull backs the top square is a homotopy push out.\\

\end{definition}

\begin{proposition}\label{cube lemma}
        Let $C$
be a cube as in \ref{cube def} where $p:E_D\to D$ be a quasi-cocell attachment in $c\mathcal{M}^{\mathcal{F}}$  with coattaching map $f$ and cocell $(cF)^{\Delta^s}$.  Assume that $B\to A$ is $0$-connected.

        Then $E_D$ is the homotopy pushout of  $ E_A \to  E_B \leftarrow E_C$ in  $c\mathcal{M}^{\mathcal{F}}$. Moreover, $j:E_B \to E_A$ is $0$-connected.\end{proposition}

  \begin{proof}  
  We may assume that all the vertical squares are pullbacks along fibrations
and   that $D\cong A\sqcup_B C$ is a pushout along two cofibrations in  $c\mathcal{M}^{\mathcal{F}}$ \cite[Theorem A.1]{Do}.
 Let us start with the case  $c\mathcal{M}^{\mathcal{F}}=c\mathcal{CA}^{\mathcal{E}}$. Recall that  pullbacks in $\mathcal{CA}$ are given by the cotensor product 
  \cite[remark p.46]{BRS} and that colimits are computed in the underlying vector spaces.  
    There is a natural map 
  \[q: E_A \oplus_{E_B}E_C\to E_D .\]

  In each cosimplcial degree $n$ there  are isomorphisms
  
 \[((cE)^{\Delta^s})^n\cong \otimes_{\sigma\in \Delta^s_n} E 
  \cong (\otimes_{\sigma\in \partial\Delta^s_n} E )\otimes  \Omega^s (E) ) \] 
In each cosimplicial degree $n$  the map $q^n$  is the chain of isomorphism

\[\scriptstyle (E_A )^n \oplus_{(E_B )^n } (E_C )^n \cong\]
\[\scriptstyle( (A^n )\square_{ ((cF)^{\partial \Delta^s })^n} ((cF)^{ \Delta^s })^n )\oplus_{((B^n )\square_{ ((cF)^{\partial \Delta^s })^n} ((cF)^{ \Delta^s })^n )}
 ((C^n )\square_{ ((cF)^{\partial \Delta^s })^n}((cF)^{ \Delta^s })^n )\cong \]
 \[\scriptstyle( (A^n )\otimes (\Omega^s (F))^n )\oplus_{( (B^n )\otimes (\Omega^s (F))^n )} ( (C^n )\otimes (\Omega^s (F))^n )\cong \]
 \[\scriptstyle ( (A^n )\oplus_{ (B^n )}  (C^n ))\otimes (\Omega^s (F))^n )\cong E_D^n\]

 We claim that 
 $j:E_B \to E_A$ is a 
 $0$-connected
  cofibration in $c\mathcal{CA}^{\mathcal{E}}$.
  
  
   We have to show  that $j^n$  is $\mathcal{E}$- injective that is to say injective for all $n$.  But $j^n$ is isomorphic to the map \[A^n\otimes  \Omega^s (E)^n \to B^n\otimes \Omega^s (E)^n  \] which is injective since $A\to B$ is. So $E_D \cong E_A \oplus_{E_B}E_C $ is a homotopy pushout.
 This completes the proof for  $c\mathcal{CA}^{\mathcal{E}}$.\\
 
 We turn to  $c\mathcal{S}^{\mathcal{G}}$.   Factor $j$ as  $E_B\stackrel{k}{\to}\bar{E}_A\stackrel{w}{\to}E_B$ in  $c\mathcal{S}^{\mathcal{G}}$ in a cofibration $k$ and a trivial fibration $w$. We have to see that the induced map  from the (homotopy)pushout  $P=\bar{E_A} \sqcup_{E_B} E_C \to E_D$ is a weak equivalence.

  Apply the homology functor to the given cube $\mathcal{C}$. The  cube $H_* \mathcal{C}$ so obtained in  $c\mathcal{CA}^{\mathcal{E}}$ has bottom square a homotopy pushout with $0$-connected  $H_*(B)\to H_* (A)$ and all vertical squares homotopy pullbacks by \ref{Hsquare} . The top square is a homotopy pushout, as we just have seen above, and the map $H_* (E_B )\to H_* (E_A )$   is $0$-connected. Hence the maps  $E_B \to E_A $ and $k$ are $0$-connected.\\
 By \ref{Hsquare} again $H_* (P)$ is a homotopy pushout. 
 It follows that the map $H_* (P)\to H_* (E_D)$ is a weak equivalence which is what we want. $\square$ \end{proof}

 

\begin{definition}\label{rescube} We say that the restricted cube axiom holds in 
if the assertion in \ref{cube def} holds for all cubes in which the map $B\to A$ is $0$-connected.
\end{definition}
 
 \begin{theorem}\label{CubeC} The restricted cube axiom holds in  $c\mathcal{M}^{\mathcal{F}}$.
 \end{theorem}
 \begin{proof} Using the fact that fibrations are retracts of quasi-cofree maps 
 the assertion follows from the coskeletal tower by means of \ref{cube lemma}.
 $\square$
 \end{proof}

We record some invariance properties  of $0$-connected maps.

\begin{lemma}\label{0pull}Let
\[ 
\xymatrix{
 E \ar[r]\ar[d]^{q} & C \ar[d]^{p} \\
        A \ar[r]   & B}\] be a homotopy pullback diagram in $c \mathcal{M}^{\mathcal{F}}$.
        If p is $0$-connected so is $q$. 
\end{lemma}
\begin{proof} The assertion for  $c \mathcal{CA}^{\mathcal{E}}$ is in \cite[Lemma 4.6.3.]{BRS}.
      For the proof in  $c \mathcal{S}^{\mathcal{G}}$, note that a map $f$ in  $c \mathcal{S}^{\mathcal{G}}$ 
is $0$-connected if and only if $H_* (f)$ is $0$-connected in  $c \mathcal{CA}^{\mathcal{E}}$ and that homotopy pullbacks are preserved under $H_*$.$\square$.\end{proof}
        
\begin{lemma}\label{0push}Let
\[ 
\xymatrix{
 B \ar[r]^{i}\ar[d]^{j} & C \ar[d]^{\bar{j}} \\
        A \ar[r]^{\bar{i}}    & P}\] be a homotopy pushout diagram in $c\mathcal{M}^{\mathcal{F}}$
        If $i$ is $0$-connected so is $\bar{i}$.\end{lemma}
        
\begin{proof} We start with the case $c \mathcal{CA}^{\mathcal{E}}$ and may assume that $i$ is a cofibration. So the map
\[i^*:[C,H_* (K(\mathbb{F},m))]\to [B,H_* (K(\mathbb{F},m))]\]
is a fibration of simplicial groups for each $m$. By the cofreenes of $H_* (K(\mathbb{F},m))$, this sequence can be identified with the induced map on the dual in degree $m$
\[i^* :(C^m)^* \to (B^m)^* .\]Because $i^*$ is $0$-connected 
it is a $\pi_0$-surjective Kan fibration and hence, surjective.
It follows that $i$ is injective. Then clearly $\bar{i}$ is injective as well since colimits are formed in the underlying cosimplicial vector spaces or equivalently in the underlying cochain complex. But then the restriction to the cocycles \[\bar{i}:Z^0 (C)=H^0 (C)\to Z^0 (P)=H^0 (P)\] is injective. \\
Suppose that the cube is in  $c \mathcal{S}^{\mathcal{G}}$. Apply $H_*$ to get a cube in $c \mathcal{CA}^{\mathcal{E}}$ which is a homotopy pushout by \ref{Hsquare}. Since $H_* (i)$ is  $0$-connected $H_* (\bar{i})$ is $0$-connected and hence so is $\bar{i}$.$\square$\\\end{proof}



\section{Applications of the cube theorem}
Let $\mathcal{N}$ be a model category and $Z$ an object of  $\mathcal{N}$. Denote the categories of objects over and under $Z$ in   
$\mathcal{N}$ by $(\mathcal{N}\downarrow Z)$ and  $(\mathcal{N}\uparrow Z)$ respectively. If the category $\mathcal{M}$ has an initial objects $\emptyset$ then $\emptyset \to Z$ is initial in  $(\mathcal{N}\downarrow X)$ and  $1_Z$ is terminal. If  $\mathcal{N}$ has a terminal object $1$ then $Z\to 1 $ is terminal in  $(\mathcal{N}\uparrow Z)$ and $1_Z$ is initial.\\ 
There are model structures on these categories in which a morphism is defined to be a weak equivalence, cofibration or fibration  if its image under the forget functor to $\mathcal{N}$ is in the corresponding class. Moreover, the property of being left proper, right proper or proper is inherited from the model category $\mathcal{N}$
\cite[Theorem 1.7.(3), Theorem 2.8.(3)]{Hi2}. Pushouts and pullbacks in $(\mathcal{N}\downarrow Z)$ and  $(\mathcal{N}\uparrow Z)$  are formed in  $\mathcal{M}$ \cite[Lemma1.3., Lemma 2.3.]{Hi2}.

\begin{definition} Let $\mathcal{N}$ be a category and $X\in \mathcal{N}$. The category objects  coaugmented under $Z$ is defined as

\[\mathcal{N}_Z:=((\mathcal{N}\downarrow X)\uparrow 1_Z)\]
\end{definition}
Objects of $\mathcal{N}_Z$ are pairs of maps $Z\to Y\to Z$ which compose to $1_Z$.
In case $\mathcal{M}$ is a (proper) model category so is  $\mathcal{M}_Z$. The object $(1_Z,1_Z)$ is initial and terminal in 
$\mathcal{N}_Z$. For short we denote it by $X$.

\begin{lemma}\label{resund} The restricted  cube axiom  holds in $(c\mathcal{M}^{\mathcal{F}}\downarrow Z)$, $(c\mathcal{M}^{\mathcal{G}}\uparrow Z)$ and in $c\mathcal{M}_Z^{\mathcal{F}}$ for each $Z\in c\mathcal{M}^{\mathcal{F}}$.
\end{lemma}

\begin{proof} This follows from \ref{CubeC} and the fact all the relevant structure is defined by the forget functor to  $c\mathcal{M}^{\mathcal{F}}$.$\square$
\end{proof}
\medskip
Recall that one says that homotopy pushouts  are universal in an $\infty$-category  if homotopy pushouts  are stable under homotopy  pullbacks.\\

\begin{lemma}\label{unipush} Homotopy pushouts along pairs of maps one of which is $0$-connected are universal
 in $c\mathcal{M}^{\mathcal{F}}$.
\end{lemma}
\begin{proof}The proof of \cite[Lemma 2.5.]{D.H} applies word for word.$\square$\\
\end{proof} 

The following result which is key for us can be found except for the items (4) and (6) in \cite{D.H}.

 \begin{theorem}\label{aC} Let $\mathcal{M}$ be an $\infty$-category with finite limits, pushouts  in which the cube axiom holds.\\
  Let $X$ and $Y$ be pointed objects in $\mathcal{M}_*$ and $F\to E\to X$ a fibration sequence. 
 Then all the assertions made in \ref{main application} hold true in  $\mathcal{M}$ where for (2) and (3) we assume in addition that countable coproducts exist in $\mathcal{M}$.$\square$\end{theorem}

 \begin{proof}
The assertions in (1), (2), (3), (5) and (7) are proved in 2.21.2, 2.10., 3.2., 2.24 and 3.5. of \cite{D.H} respectively.
  For (4) note that by \cite[Lemma 2.5.]{D.H} pushouts are universal in $\mathcal{M}$.  Point (4) follows from an application of this universality to the diagram \[ 
\xymatrix{
 \ast \ar[r]\ar[d] & X\ar[d]  \\
        X \ar[r]   & X\vee X}\]
        and the two maps $ X\vee X\stackrel{(1,1)}{\to} X$ and $\ast \to X$. For (6) apply universality of pushouts to 
        \[ 
\xymatrix{
 \ast \ar[r]\ar[d] & X\ar[d]  \\
        Y \ar[r]   & Y\vee X}\]
        and the two maps $ Y\vee X\stackrel{(1,\ast )}{\to} Y$ and $\ast \to X$.
         $\square$
 \end{proof}
 
 \begin{remark}In the context of model categories, (3) and (5) were proved by Doerane in \cite{Do2} as we learned from \cite{R.S.} where  a version of (4) is given.
 \end{remark}
 
   Now we apply \ref{aC} to the model category $c\mathcal{M}_Z^{\mathcal{F}}$.
 Since it is  model category enriched over pointed simplicial sets
 $\mathcal{S}_*$ the
   categories of cofibrant and fibrant objects  define pointed $\infty$-categories in the model of simplicial categories.
    We denote the underlying pointed $\infty$-categories  by 
 $c M_Z^{\mathcal{F}}$.  They have all small limits and colimits. The restricted cube axiom holds in     
 $(c M^{\mathcal{F}})_Z$ by \ref{CubeC} and \ref{resund}.
\\

 \begin{proof} of (\ref{main application}).\\
One simply checks that under the connectivity assumptions  made the proofs in \cite{D.H} go through in the $\infty$-category  $c M_Z^{\mathcal{F}}$ since the restricted cube axiom is sufficient  in all applications of the cube axiom made. For (4) and (6) one uses \ref{unipush}.$\square$
\end{proof}

 \section{Connectivity estimates }
 
 In this section we provide the results on cosimplicial connectivity needed in the proof of the Hilton-Milnor theorem. The proofs flow from three sources. First the fact that colimits in $\mathcal{CA}_Z^{E}$ are formed in the underlying cosimplicial vector spaces. Second \ref{Hsquare} which enables us to reduce the proofs in  $\mathcal{S}_Z^{G}$ to the ones in $\mathcal{CA}_Z^{E}$. Third the dual Blakers-Massey theorem which replaces the generalized Blakers-Massey theorem of \cite{A.B.F.J.} valid in $\infty$-topoi.\\
 
Recall that finite products in $\mathcal{CA}_Z^{E}$ are defined by the cotensor product over $Z$. Infinte products are described as the filtered limit  of all finite subproducts. Now filtered limits are given as the colimit of all unstable subcoalgebras of finite dimension in the limit of the underlying vector spaces. So all in all  
 \[\square_I C_i = colim_{D_a } lim^{Vec}_J (\square )_J C_j\]
 where $J$ runs through the finite subsets of $I$ and the colimit over the finite dimensional subcoalgebras $D_a$ in the vector space limit.
 This may be seen as  in the absolute case to be found for example in \cite[1.1.b]{G.4}. Another way to the construction of infinite products is via the cofree functor right adjoint to the forget functor $I$:
   \[I:\mathcal{CA}_C \rightleftarrows (Vec\uparrow J(C)):G_{C}\] 
   as in \cite{A}.\\
 
 \begin{proposition}\label{inf prod} Let $I$ be a discrete category and $D_i \in c\mathcal{M}_Z^{\mathcal{F}}$ fibrant and $D_i \to Z$  s-connected for all i. 
  Then $\square_I D_i \to Z$ is s-connected.
 \end{proposition}
 \begin{proof} First note that finite products $c\mathcal{CA}_C$ are defined by the degreewise cotensor product over $C$ and this functor has the asserted property by the spectral sequence \cite[Theorem 4.5.1 (b)]{BRS}. Since $D_i$ is fibrant there is an isomorphism
 \[D_i^n \cong Z^n \otimes E_i^n \] with $E_i \in \mathcal{E}$
 and consequently an isomorphism \[\square_{Z^n} (D^n_{i} ,D^n_{j})\cong E^n_i \otimes E_j^n \otimes Z^n .\] So the limit under consideration is just an ordinary infinite tensor product. The structure maps of the system in degree $n$ are given by the projections
 \[E^n_{i_1}\otimes \ldots\otimes E_{i_k} ^n \otimes Z^n \to E^n_{i_1}\otimes \ldots \otimes \hat{E^n_{i_s}}\otimes \ldots \otimes E_{i_k} ^n \otimes Z^n \] which are clearly surjective.
 By \cite[Proposition 3.4.3]{BRS} we may assume that $E_i^n=\mathbb{F}$ for $n\leq s$. Consequently $D_i^n =Z^n$ for $n\leq s$.
 By \cite[p.84]{We} there is a short exact sequence for each r 
 \[Lim^{Vec,1}\pi^{r-1}(\otimes_{J}E_j^*\otimes Z^* )\to\pi^r (Lim^{Vec}(\otimes_{J}E_j^*\otimes Z^* )\to Lim^{Vec}\pi^r (\otimes_{J}E_j^*\otimes Z^* ).\] For $r-1\leq s$ the derived functor is trivial since  in this case the system is constant.
 So there are isomorphisms for $r\leq s$
 \[\pi^r (Lim^{cVec}(\otimes_{J}E_j\otimes Z ))\cong Lim^{Vec}\pi^r (\otimes_{J}E_j\otimes Z )\cong \pi^r (Z)\]
 The actual limits in coalgebras are then defined by the colimits over the filtered system  of finite dimensional subcoalgebras. 
 Colimits are computed in the underlying cosimplicial vector spaces or equivalently cochain complexes.  Hence, this colimit commutes with $\pi^*(-)$. This implies the assertion in $c\mathcal{CA}^{\mathcal{E}_C}$  because every unstable coalgebra is the colimit of its finite subcoalgebras .\\
  The claim in $c\mathcal{S}^{\mathcal{G}}$ may be reduced to the case just proved as follows. By the same argument as above the  products which show up are the ordinary  products of the form $ (\times_{i\in I} G_i^n ) \times Z^n$ with $G_i \in \mathcal{G}$. We have to estimate the connectivity of the functor $\pi^* H_{\ast}$ applied to these products.
  Now recall the  fact that homology commutes with infinite products of Eilenberg-MacLane spaces of type $\mathbb{F}$  and combine this with the assertion in  $c\mathcal{CA}^{\mathcal{E}}$ just seen. For finite characteristic this fact is a consequence of \cite[4.4.]{Bou3} and in characteristic zero it may be deduced from the Milnor-Moore theorem as in the proof of \cite[Theorem 4.2.2.]{BRS}. This completes  the proof.$\square$ \\
 \end{proof}
 
 \begin{lemma}\label{convee} Let
  $X_n\in c\mathcal{M}_Z^{\mathcal{F}}$ with \[q_n :X_n\to Z\]  $k_n$-connected 
  such that $k_n\geq 0$ for all $n\in \mathbb{N}$. Then $\vee X_n \to Z$ is at least $m$-connected with $m=min(k_n |n\in \mathbb{N})$.
\end{lemma}
\begin{proof} We may assume that the maps $i_n :Z\to X_n$  are cofibrations.  Consider the case of
$c\mathcal{CA}^{\mathcal{E}}_Z$ first.  The maps $i_n :Z\to X_n$ are injective as they are $0$-connected cofibrations.
Consider the exact sequence of cochain complexes for $X_1 , X_2$.
\begin{equation}\label{eqex}0\to Z\stackrel{(i_1 , -i_2 )}{\longrightarrow}  X_1 \oplus X_2\to X_1 \oplus_Z X_2 \to 0 \end{equation}
Because \[H^{*}(i_j ):H^{*}(Z)\to H^{*}(X_j)\] is injective for $j\leq 2$ the long exact cohomology sequence of \ref{eqex} of decomposes in short exact sequences.
Since  $H^{ k_j\leq}(i_j) $ is an isomorphism it follows that there is an isomorphism 
\[(i_1 \oplus i_2 )^{*}:H^{\leq m}(Z)\cong H^{\leq m}(X_1 \oplus_Z X_2) .\]  By an easy  induction the assertion holds for any finite subset of $\mathbb{N}$.\\
Now there is an isomorphism \[\oplus X_n \cong colim_{|J|<\infty}(\oplus)_J {X_j}\] and since filtered colimits are exact the assertion holds in $c\mathcal{CA}^{\mathcal{E}}_Z$.\\
We turn to the case $c\mathcal{S}^{\mathcal{G}}_Z$.
 By \cite[Proposition 6.1.7]{BRS} the statement for finite subsets of the natural numbers hols there as well. The infinite coproduct 
 $c\mathcal{S}^{\mathcal{G}}_Z$ is again the colimit of its finite subcoproducts and since singular homology commutes with filtered colimits of cofibrations the assertion holds in $c\mathcal{S}^{\mathcal{G}}_Z$.$\square$\\
\end{proof}

We turn to smash powers now. The next result is a consequence of the homotopy excision theorem in $c\mathcal{M}^{\mathcal{F}}$\cite[Theorems 4.6.5, 6.2.6.]{BRS}.
 
 \begin{lemma}\label{conest2} Let
   $X,Y\in c\mathcal{M}_Z^{\mathcal{F}}$   with $X\to Z$ and $Y\to Z$  $k$ and $l$ connected   respectively.
 Then $X\wedge Y$ is $k+l+1$-connected as a pointed object. \end{lemma}
 \begin{proof} We consider the case $c\mathcal{CA}_Z^{\mathcal{E}}$. Note that in the abelian category of coaugmented cosimplicial $Z$-comodules $cComod(Z)_Z$ coproducts and products coincide. Consider
 \[X\oplus Y\stackrel{i}{\to} X\square Y \stackrel{q}{\to} X\oplus Y\]
 where the  map on the left is the canonical inclusion and the map on the right is the canonical map of $Z$-comodules. The composition is the identity and $q$ is $k+l+1$-connected by \cite[4.6.5. b)]{BRS}. It follows that $i$ is $k+l+1$-connected as well. By the homotopy pushout square
   \[ 
\xymatrix{
 X\oplus Y \ar[r]\ar[d] & X\square Y\ar[d]  \\
        Z \ar[r]   & X\wedge Y}\]
       $ X\wedge_Z $ is $k+l+1$-connected as a pointed object which was to be seen.\\
       We consider  the case $c\mathcal{S}^{\mathcal{G}}$. Since $H_*$ commutes with homotopy pullbacks \cite[6.2.3.]{BRS} there are isomorphisms \[H_* ( X\times Y)\cong H_* (X)\square_{}H_{*}(Y).\]
       Since  $H_*$ commutes with homotopy pushouts along pairs of maps one of which is $0$-connected \ref{Hsquare},  there are isomorphisms
       \[H_* ( X \vee Y)\cong H_* (X)\oplus_{}H_{*}(Y).\]
       Apply the same fact again to the homotopy pushout square
       
         \[ 
\xymatrix{
 X\vee Y \ar[r]\ar[d] & X\times Y\ar[d]  \\
        Z \ar[r]   & Y\wedge X}\]
  and find an isomorphism      
  \[ H_* (Y\wedge X)\cong H_* (X)\wedge H_* (Y).
  \] Now the assertion follows from the proved statement in  $c\mathcal{CA}^{\mathcal{E}}_Z$.$\square$
       
  \end{proof}


\begin{lemma}\label{omegacon} Let $X\in  c\mathcal{M}^{\mathcal{F}}_Z$ and assume that $X\to Z$ is $k$-connected. Then $\Omega \Sigma (X)\to Z$ is also $k$-connected.
\end{lemma}

\begin{proof} We  start as usual by considering  $c\mathcal{CA}^{\mathcal{E}}_Z$. The assertion is true in the abelian category $cComod(Z)_Z$ since the loop and suspension functors  $\Omega^{Comod}$ and $\Sigma^{Comod} = \Sigma$ are inverse to each other.
By \cite[Theorem 4.6.5. b)]{BRS} the natural map
\[\Omega  \Sigma (X)\to \Omega^{Comod}\Sigma(X)\]
is $(2k+1)$-connected.  This proves the assertion in $c\mathcal{CA}^{\mathcal{E}}_Z$.\\
 It holds in $c\mathcal{S}^{\mathcal{G}}_Z$
 by an application of \ref{Hsquare}.$\square$
\end{proof}
 
 \section{ The Hilton-Milnor theorem in   $c\mathcal{M}^{\mathcal{F}}_Z$}
We start with a preparatory lemma  established by Gray in the category of topological spaces \cite{Gr}. 

\begin{lemma}\label{fibproj}Let
  $X,Y\in c\mathcal{M}^{\mathcal{F}}_Z$ with $X\to Z$ and $Y\to Z$ $0$-connected. Then there is an equivalence \[ Y\times \Sigma X/ Y\simeq \Sigma X\vee \Sigma(X \wedge Y)\]
\end{lemma}
\begin{proof} The following diagram
\[ 
\xymatrix{
 X\times Y \ar[r]\ar[d] & Y \ar[d] \\
        Y \ar[r]   & \Sigma X\times Y}
     \] is a homotopy pushout by \ref{unipush} as is its composition with
     
\[ 
\xymatrix{
  Y \ar[r]\ar[d] &  \Sigma X\times Y \ar[d] \\
        \ast \ar[r]   & \Sigma X\vee \Sigma(X \wedge Y)}
     \]
     by \cite[Corollary 2.21.1]{D.H}. The pasting law for homotopy pushouts implies that the second diagram is a homotopy pushout as well.$\square$
\end{proof} 

\begin{corollary} Let $\Sigma X, Y\in  c\mathcal{M}_Z^{\mathcal{F}}$ be $1$-connected as pointed objects. Then there is a homotopy fiber sequence 
\[\Sigma X\vee \Sigma (X\wedge \Omega  Y)\to\Sigma X\vee Y\to  Y\]
\end{corollary}
\begin{proof} This follows from \ref{fibproj} together with (6) in \ref{main application} .$\square$\\\end{proof}

Using this we arrive at.

\begin{lemma} Let $\Sigma X, Y\in  c\mathcal{M}_Z^{\mathcal{F}}$ be $1$-connected as pointed objects. Then there is an equivalence
\[\Omega (\Sigma  X\vee Y)\simeq \Omega (Y)\times \Omega (\Sigma  X\vee \Sigma (X\wedge \Omega  Y))\]
\end{lemma}\begin{proof} The equivalence is induced by the inclusion  
$  Y\to\Sigma X\vee Y$ and the map  $\Sigma X\vee \Sigma (X\wedge \Omega  Y)\to\Sigma X\vee Y$. The loops of these maps can be multiplied in $\Omega (\Sigma  X\vee Y)$ which is a group object in the homotopy category of  $c\mathcal{M}_Z^{\mathcal{F}}$ and this map is the searched for equivalence.\\ $\square$
\end{proof}

As a consequence of the James splitting in \ref{main application} we get:

\begin{lemma}\label{conssplitting} Let $\Sigma X, \Sigma Y\in  c\mathcal{M}_Z^{\mathcal{F}}$ be $1$-connected as pointed objects. Then there is an equivalence 
\[\Omega (\Sigma  X\vee \Sigma  Y)\simeq \Omega \Sigma ( X)\times  \Omega \Sigma (\bigvee_{0\leq i<\infty} Y\wedge X^{\wedge i})\]
\end{lemma}

The proof of the Hilton-Milnor   theorem below proceeds from here on essentially as in the classical sources \cite{Mi}, \cite{Wh} but amplified by the connectivity estimates proved in section 5.
  These are needed in order to control the convergence of a series of maps which define the asserted equivalence.\\

We recall the classical  notion of basic products on $k$ generators.

\begin{definition} Let $T$ be the free non-associative algebra on  $k$ generators $x_1 ,\ldots ,x_k$. Basic products of weight $w$ and rank $r$  are given inductively. We define  $x_1 ,\ldots ,x_k$ to be the basic products of weight $1$ and rank $0$. Suppose the basic products of weight $n$ have been found as $x_1, \ldots ,x_m$ together with a rank function  $r$ which satisfies $r(x_i)<i$. The basic products of weight $n+1$ are formed as all $x_i x_j$ with $w(x_i)+w(x_j)=n+1$ and $r(x_i)\leq j<i$. Order them in any way $x_{m+1},\ldots ,x_{l}$ and extend the function $r$  as $r(x_h)=j$ for $x_h = x_i x_j$.
\end{definition} 

For more details we refer to \cite{Wh}. By a theorem of Hall \cite{Ha} the basic products form a basis for the free Lie ring on generators $x_1 ,\ldots ,x_k$.
 
 \begin{definition} Let $w=x_m$ be the m-th basic product. Define  \[w(X,Y)=X_m  = X^{a_x }\wedge Y^{a_y }\] where $a_x$ is the number of occurrences of $x$ in the word $x_m$ and $a_y$ the number of occurrences of $y$.
 \end{definition}

 

 \begin{proof}of the Hilton-Milnor theorem \ref{traditional HM}.\\
  We define \[R_m =\bigvee_{\stackrel{i\geq m}{r(x_i )<m}} X_i , \quad R^{'}_m =\bigvee_{\stackrel{i>m}{r(x_i )<m}} X_i\],
 so that $R_m = X_m \vee R^{'}_m$. We have equivalences
 \[\Omega \Sigma (X_m \vee R^{'}_m )\simeq \Omega \Sigma (X_m)\times \Omega \Sigma ( \bigvee_{{i\geq 0}}( R^{'}_m\wedge  X_m^{\wedge i})\simeq \]\[ \Omega \Sigma (X_m)\times \Omega \Sigma ( \bigvee_{i\geq 0}\bigvee_{\stackrel{j>m}{r(w_j )<m}}(X_j \wedge X_m^{\wedge i})\simeq  \Omega \Sigma (X_m)\times  \Omega \Sigma (R_{m+1})\] by  \ref{conssplitting} and definition of $R^{'}_m$ 
where  the last equivalence holds by a combinatorical identity for basic products to be found in \cite[6.1. p.512]{Wh}.\\
For $x_m$ a basic monomial of weight $w$ the connectivity of  $X_m  $ is bigger or equal to $w-1$ by   \ref{conest2}. As there are only finitely many basic  monomials of a given weight, the connectivity of $X_m$ tends to infinity with $m$.\\
 Due to \ref{convee} the connectivity of $R_m$ tends with $m$ to infinity as well.\\ Let \[f_m: \Omega \Sigma (X_m ) \times  \Omega \Sigma(R_{m+1})\to  \Omega \Sigma(R_{m})\] be  a weak equivalence.
Let $g_m = 1\times f_m$ then
\[g_m:(\Omega \Sigma (X_1 )\times  \ldots \times \Omega \Sigma (X_{m-1} ))\times \Omega \Sigma (X_m ) \times  \Omega \Sigma(R_{m+1})\to (\Omega \Sigma (X_1 )\times  \ldots \times \Omega \Sigma (X_{m-1} ))\times  \Omega \Sigma(R_{m})\]
is a weak equivalence as follows  from the spectral sequence of \cite[Theorem 4.5.1 (b)]{BRS}.\\
 Let $h_m = g_1 \circ g_2 \circ \ldots \circ g_m  |:\Omega \Sigma (X_1 )\times \ldots \times \Omega \Sigma (X_m )$.\\
 Because the pointed object $P=\prod_{\stackrel{m=1}{}}^{\infty} \Omega \Sigma (X_m )$ is filtered by the objects $P_n =\prod_{\stackrel{i=1}{}}^n  \Omega \Sigma (X_i )$ and $h_n |P_{n-1}=h_{n-1}$   the collection of  $h_m$ define a map
\[h:P \to \Omega \Sigma (R_1)=\Omega \Sigma (X\vee Y ).\]  
We show that $h$ induces an isomorphism on $\pi^{r}()$ for each $r$.
Choose $N$  such that $X_m$ and $R_{m+1}$ are $r$-connected for all $m\geq N$. Consider the obvious maps
\[i_N :P_N \to P\] and 
\[h_N :P_N \to P_N \times \Omega \Sigma (R_{N+1}).\]

 As a consequence of \ref{inf prod}, \ref{omegacon} and an application of the K\"unneth spectral sequence , they induce  isomorphisms on $\pi^r(-)$.
 Inspection of the commutative diagram
 
 \[ 
\xymatrix{
 \pi^r (P_N ) \ar[r]^-{h^r_N}\ar[d]^{i^r_N} &  \pi^r (P_N \times \Omega \Sigma (R_{n+1}) ) \ar[d]^{g^r_N } \\
       \pi^r(P) \ar[r]^-{h^r}   & \pi^r (\Omega \Sigma (X\vee Y) )}
     \] gives the assertion.$\square $ 
 
 \end{proof}

\begin{remark} Porter has extended the classical Hilton-Milnor theorem by decomposing the loop space of the fat wedge of suspension spaces \cite{P}. Such an extension also holds in $c\mathcal{M}_Z^{\mathcal{F}}$. 
\end{remark}

\noindent
 Manfred Stelzer,
Universit\"at Osnabr\"uck.

\end{document}